\documentclass{amsart}
\usepackage{graphicx}
\newtheorem{thm}{Theorem}[section]
\newtheorem{cor}[thm]{Corollary}
\newtheorem{lem}[thm]{Lemma}
\newtheorem{prop}[thm]{Proposition}
\theoremstyle{definition}

\theoremstyle{remark}

\newtheorem{question}[thm]{Question}
\numberwithin{equation}{section}
\newcommand{\ve}{\varepsilon}

\begin{document}

\title[ $J$-flow on K\"{a}hler surfaces]{The $J$-flow on K\"{a}hler surfaces: a boundary case$^1$}%
\author{Hao Fang}%
\address{Hao Fang, Department of Mathematics, University of Iowa}%
\email{hao-fang@uiowa.edu}%
\author{Mijia Lai}
\address{Mijia Lai, Department of Mathematics,
University of Rochester}%
\email{lai@math.rochester.edu}%
\author{Jian Song}
\address{Jian Song, Department of Mathematics,
Rutgers}
\email{jiansong@math.rutgers.edu}
\author{Ben Weinkove}
\address{Ben Weinkove, Mathematics Department,
Northwestern University}
\email{weinkove@math.ucsd.edu}

\thanks{}%
\footnotetext[1]{Research supported in part by NSF grants DMS-1008249, DMS-08047524 and DMS-1105373.  This work was carried out while the fourth-named author was a member of the Mathematics Department at UC San Diego.}

\begin{abstract}
We study the $J$-flow on K\"{a}hler surfaces when the K\"ahler class lies on  the boundary of the open cone for which global smooth convergence holds, and satisfies a nonnegativity condition.
We obtain a $C^0$ estimate and show that the $J$-flow converges smoothly to a singular K\"{a}hler metric away from a finite number of  curves of negative self-intersection on the surface.  We discuss an application to the Mabuchi energy functional on K\"ahler surfaces with ample canonical bundle.
\end{abstract}
\maketitle
% ----------------------------------------------------------------
\section{Introduction}

The $J$-flow is a parabolic flow on K\"{a}hler manifolds with two K\"{a}hler classes. It was defined by Donaldson~\cite{Do} in the setting of moment maps and by Chen~\cite{Ch1} as the gradient flow of the $\mathcal{J}$-functional appearing in his formula for the Mabuchi energy~\cite{M}.

The $J$-flow is defined as follows. Let $X$ be a compact K\"{a}hler manifold with two K\"{a}hler metrics $\omega$ and $\chi$ in different K\"{a}hler classes $[\omega]$ and $[\chi]$ respectively. Let $\mathcal{P}_{\chi}$ be the space of smooth  $\chi$-plurisubharmonic functions on $X$.  Namely,
\[
\mathcal{P}_{\chi}=\{ \varphi\,|\, \chi_{\varphi}:=\chi+dd^{c}\varphi>0\}.
\]
Then the $J$-flow is a flow defined in $\mathcal{P}_{\chi}$ by
\begin{align} \label{j-flow}
%\left\{
%\begin{array}{l l}
\frac{\partial}{\partial t} \varphi &=c-\frac{ n\chi_{\varphi}^{n-1}\wedge \omega}{\chi_{\varphi}^n}, \qquad
%\\
\varphi(0) =\varphi_{_0}\in \mathcal{P}_{\chi},
%\end{array}
%\right.
\end{align}
where $c$ is the topological constant given by
\[
c=\frac{n[\chi]^{n-1}\cdot[\omega]}{[\chi]^n}.
\]
A stationary point of (\ref{j-flow}) gives a critical K\"{a}hler metric $\tilde{\chi} \in [\chi]$ satisfying
\begin{align} \label{critical}
c \tilde{\chi}^{n}= n\tilde{\chi}^{n-1}\wedge \omega.
\end{align}

Donaldson noted that a smooth critical metric  exists only if the cohomological condition $[c \chi - \omega]>0$ holds  \cite{Do}.  In complex dimension 2, Chen \cite{Ch1} showed that this necessary condition is sufficient for the existence of a smooth critical metric, by observing that in this case (\ref{critical}) is equivalent to the complex Monge-Amp\`ere equation solved by Yau \cite{Y} (see (\ref{dma}) below).  Chen also established the long time existence for the $J$-flow (\ref{j-flow}) with any initial data \cite{Ch2}.
The fourth-named author \cite{W1,W2} showed that the $J$-flow converges to a critical metric if the cohomological condition $[c\chi -(n-1)\omega]>0$ holds.  In particular, if $X$ is a K\"ahler surface, a necessary and sufficient condition for convergence of the flow to a smooth critical metric is Donaldson's cohomological condition $[c \chi - \omega]>0$.
  
In \cite{SW}, the third and fourth-named authors found a necessary and sufficient condition for the convergence of the $J$-flow in higher dimensions, which we now explain. Define
\begin{align} \label{2}
\mathcal{C}_{\omega}:=\{ \, [\chi]>0 \, \, |\, \exists \chi'\in [\chi]\quad \text{such that} \quad c\chi'^{n-1}-(n-1)\chi'^{n-2}\wedge \omega>0\}.
\end{align}
Then it was shown in \cite{SW} that 
 the $J$-flow (\ref{j-flow}) converges smoothly to the critical metric solving (\ref{critical}) if and only if $[\chi]\in \mathcal{C}_{\omega}$.

%Ever since its debut, the $J$-flow has been studied by several authors ~\cite{Ch1, Ch2, W1, W2}, culminating in~\cite{SW}, where the authors provide a necessary and sufficient condition for $J$-flow to converge smoothly to the critical metric satisfying (\ref{critical}).

%We explain their condition. By choosing a normal coordinate system for $\omega$ so that $\chi$ is diagonal with entries $\lambda_1, \cdots, \lambda_n$, the critical equation (\ref{critical}) becomes
%\begin{align}
%\sum_{i=1}^{n}\frac{1}{\lambda_i}=c .
%\end{align}
%Their condition can be phrased as
%\begin{align}
%\sum_{i=1,i\neq k}^{n} \frac{1}{\lambda_i} <c \quad \text{for} \quad k=1,2, \cdots, n,
%\end{align}
%which can be translated into a condition of a certain $(n-1,n-1)$ form being positive.
%
%More precisely, define

 In \cite{FLM, FL1}, the $J$-flow has been generalized to the general inverse $\sigma_k$ flow. An analogous necessary and sufficient condition is found to ensure the smooth convergence of the flow.

The behavior of the $J$-flow in the case when the condition $[\chi]\in \mathcal{C}_{\omega}$ does \emph{not} hold is still largely open.  However, recent progress has been made by the first and second named authors \cite{FL2} in the case of a family of K\"ahler manifolds satisfying the Calabi symmetry condition.   It was shown (in the more general case of the inverse $\sigma_k$ flow) that if the initial metric satisfies the Calabi symmetry, the flow converges to a K\"ahler current which is the sum of a K\"ahler metric with a conic singularity and a current of integration along a divisor.

We consider the case when $X$ is a  K\"ahler surface. As discussed above, a necessary and sufficient condition for convergence of the flow to a smooth critical metric is:
\begin{align}
[c\chi-\omega]>0.
\end{align}
Donaldson \cite{Do} remarked that if this condition fails, then one might expect the $J$-flow to blow up over some curves of negative self-intersection.
 It was observed in \cite{SW} (see Proposition 4.5 there) that, applying the results of 
 Buchdahl~\cite{Bu} and Lamari~\cite{Lam}, there exist a finite number $N \ge 0$, say, of irreducible curves $C_i$ with $C_i^2<0$ on $X$ and positive real numbers $a_i$ such that $[c\chi-\omega]-\sum_{i=1}^{N} a_i [C_i]$ is K\"{a}hler.   Moreover, it was shown in \cite{SW} that, at least for some sequence of points approaching some $C_i$ that the quantity $|\varphi| + | \Delta_{\omega} \varphi|$ blows up. 
 
In this paper we describe the behavior of the $J$-flow for certain  classes $[\chi]$ on the boundary of $\mathcal{C}_{\omega}$.  First we introduce some notation:  given a closed (1,1)-form $\alpha$, write $[\alpha] \ge 0$ if there exists a smooth closed nonnegative (1,1) form cohomologous to $\alpha$.   We consider any K\"ahler class $[\chi]$ satisfying:
\begin{equation} \label{clos}
[c\chi - \omega] \ge 0.
\end{equation}
All such classes $[\chi]$ lie in the closure of $\mathcal{C}_{\omega}$.  Note that the boundary of $\mathcal{C}_{\omega}$ consists of K\"ahler classes $[\chi]$ such that $[c\chi - \omega]$ is \emph{nef}, which means that for every $\ve>0$ there exists a representative of $[c\chi-\omega]$ which is bounded below by $-\ve \omega$.
  Further, since \[
[c\chi-\omega]^2=[\omega]^2>0,
\] the class $[c \chi-\omega]$ is nef and big.  Nevertheless,  to our knowledge, this does not imply that it satisfies (\ref{clos}) - see Question \ref{q1} below.
However, at least in many cases the condition (\ref{clos}) is equivalent to $[\chi]$ belonging to the closure of $\mathcal{C}_{\omega}$ in the K\"ahler cone.  This holds for all Hirzebruch surfaces, for example, since explicit nonnegative $(1,1)$ forms can be found representing all classes on the boundary of the K\"ahler cone (see the discussion in \cite{Ca}).

Our main result is:

\pagebreak[3]
\begin{thm} [Main Theorem] \label{main} Let $X$ be a compact K\"ahler surface with K\"ahler metrics $\omega$ and $\chi$ such that 
$$[c\chi - \omega] \ge 0, \quad \textrm{where } c= \frac{2 [\chi] \cdot [\omega]}{[\chi]^2}.$$
Then there exist a finite number of curves $C_i$ on $X$ of negative self-intersection such that the solution $\varphi(t)$ of the $J$-flow (\ref{j-flow}) converges in $C^{\infty}_{\emph{loc}}(X\setminus \bigcup C_i)$ to a continuous function $\varphi_{\infty}$, smooth on $X\setminus \bigcup C_i$, satisfying
\begin{align} \label{criticaleqn}
c \chi_{\varphi_{\infty}}^2 =2\chi_{\varphi_{\infty}} \wedge \omega, \quad \textrm{for} \quad  \chi_{\varphi_{\infty}} = \chi + dd^c \varphi_{\infty} \ge 0.
\end{align}
Moreover, $\varphi_{\infty}$ is the unique continuous solution of (\ref{criticaleqn}) up to the addition of a constant.
\end{thm}

Our result makes use of some recent works in the study of complex Monge-Amp\`ere equations that appeared after the breakthrough of  Ko{\l}odziej \cite{K1}.  Indeed the existence of a unique weak solution to the critical equation (\ref{criticaleqn}) is a direct consequence of a result of
Eyssidieux-Guedj-Zeriahi~\cite{EGZ1} and Zhang~\cite{Z}, who generalized Ko{\l}odziej's theorem to the degenerate complex Monge-Amp\`ere equation.  By comparing with this solution, we obtain our key uniform estimate for $\varphi(t)$ along the $J$-flow (Proposition \ref{p1} below).  In addition, we use the viscosity methods introduced in \cite{EGZ2} to give a second proof of our key estimate.  Moreover, the results of \cite{EGZ2} allow us to conclude that the solution of (\ref{criticaleqn}) above is continuous, and also that (\ref{criticaleqn}) can be understood in both the pluripotential and the viscosity senses.

We have an application of our result to the \emph{Mabuchi energy} \cite{M}, a functional which is closely connected to the problem of algebraic stability and existence of constant scalar curvature K\"ahler (cscK) metrics  \cite{Y2, T1, D2}.  Given a K\"ahler surface $(X, \chi)$, the Mabuchi energy is the functional $\textrm{Mab}: \mathcal{P}_{\chi} \rightarrow \mathbb{R}$ given by
$$\textrm{Mab}(\varphi) = - \int_0^1 \int_X \frac{\partial \varphi_t}{\partial t} (R_{\chi_{\varphi_t}} - \mu) \chi_{\varphi_t}^n dt,$$
where $\{ \varphi_t \}_{0 \le t \le 1}$ is a path in $\mathcal{P}_{\chi}$ between 0 and $\varphi$, $R_{\chi_{\varphi_t}}$ is the scalar curvature of the metric $\chi_{\varphi_t}$ and $\mu$ is the average of the scalar curvature of $\chi$.  The value $\textrm{Mab}(\varphi)$ is independent of the choice of path.

It was conjectured by Tian \cite{T1}, assuming $X$ has no nontrivial holomorphic vector fields, that the existence of a cscK metric is equivalent to the \emph{properness} of the Mabuchi energy, meaning that there exists an increasing function $f: [0,\infty) \rightarrow \mathbb{R}$ with $\lim_{x \rightarrow \infty} f(x)=\infty$ such that
$$\textrm{Mab}(\varphi) \ge f(E(\varphi)), \quad \textrm{where } \quad E(\varphi) = \int_X \sqrt{-1} \partial \varphi \wedge \overline{\partial}\varphi \wedge (\chi_0 + \chi_{\varphi}).$$
This conjecture holds whenever $[\chi] = -c_1(X)>0$ or if  $[\chi]=c_1(X)>0$ and $X$ has no nontrivial holomorphic vector fields \cite{T1,T2, TZ}.  It also holds on all manifolds with $c_1(X)=0$ even in the presence of holomorphic vector fields \cite{T2}.  In fact in each case, the function $f$ can be taken to be linear \cite{T2, PSSW}.  Chen \cite{Ch1}  showed that on manifolds with $c_1(X)<0$, or equivalently, with ample canonical bundle $K_X$, the Mabuchi energy can be written as a sum of two terms:  the first is the $\mathcal{J}$-functional with reference metric $\omega$  in $[K_X]$ and the second is a term which is bounded below.  In fact the second term is proper \cite{T2} (see the discussion in \cite{SW}), and under the cohomological condition $[c\chi - \omega] \ge 0$, the $\mathcal{J}$ functional has a lower bound, as shown in Corollary \ref{c1} below.  Hence we obtain: 

\begin{cor} \label{t2}Suppose that $X$ is a compact K\"{a}hler surface with ample canonical bundle $K_X$. Then the Mabuchi energy is proper on the classes $[\chi]$ satisfying:
\begin{equation} \label{cone}
\left( \frac{2[\chi]\cdot[K_X]}{[\chi]^2} \right)[\chi]-[K_X]\geq 0.
\end{equation}
Moreover, the function $f$ in the definition of properness can be taken to be linear.
\end{cor}

Thus, since the condition of $K_X$ being ample implies that $X$ has no nontrivial holomorphic vector fields,  conjecturally, classes $[\chi]$ in the cone given by (\ref{cone}) should admit cscK metrics.  Note that the class $[K_X]$ is inside this cone and admits a cscK metric \cite{A, Y}.  The same is true for classes sufficiently close to $[K_X]$ (see \cite{LS}).  On the other hand, Ross found K\"ahler classes on surfaces with $K_X$ ample that do not admit cscK metrics \cite{R}.  Corollary \ref{t2}, together with the arguments of \cite{LS}, suggests that the set of classes that admit cscK metrics is strictly larger than those lying in the cone (\ref{cone}).

An outline of the paper is as follows.   In Section \ref{sect:main}, we prove the key $C^0$ estimate.  We provide two proofs:  the first uses smooth maximum principle arguments and the second uses the notion of viscosity solutions from \cite{EGZ2}.  We complete the proof of the main theorem in Section \ref{sect:thm}, and in the last section we finish with some questions for further study.

We thank Valentino Tosatti for some helpful suggestions and discussions.  We also thank  the referees for their comments which helped to improve the exposition, and for pointing out some inaccuracies in a previous version of this paper.

\section{The $C^0$ Estimate} \label{sect:main}

For convenience of notation we assume from now on that $c=1$.  We may do this by considering $\frac{1}{c}[\chi]$ instead of $[\chi]$. In addition, we may assume, by modifying the initial data if necessary, that $\chi - \omega \ge 0$.

The key estimate we need is a uniform $C^0$ estimate for the solution $\varphi(t)$ of the $J$-flow.   We need the following theorem on the degenerate complex Monge-Amp\`ere equation  \cite{EGZ1, EGZ2} (the $C^0$ estimate was proved independently in \cite{Z} under slightly less general hypotheses). 

\begin{thm} \label{thm:EGZ1}
Let $(M, \omega)$ be a compact K\"ahler manifold of complex dimension $n$ and let $\alpha$ be a semi-positive $(1,1)$ form with $\int_M \alpha^n>0$.  For any nonnegative $f \in L^p (M, \omega^n)$, for $p>1$, with $\int_M f \omega^n = \int_M \alpha^n$ there exists a unique continuous function $\varphi$ on $M$ with $\alpha + dd^c \varphi \ge 0$ and
\begin{equation}
(\alpha+ dd^c \varphi)^n = f \omega^n, \quad \sup_M \varphi =0.
\end{equation}
Moreover, $\| \varphi\|_{C^0(M)}$ is uniformly bounded by a constant depending only on $p, M, \omega, \alpha$ and $\| f\|_{L^p(M)}$.
\end{thm}

Given this, we immediately obtain a solution $\varphi_{\infty}$ to (\ref{criticaleqn}), using the observation of Chen \cite{Ch1} that the critical equation can be rewritten as a complex Monge-Amp\`ere equation:
\begin{align} \label{dma}
\chi_{\varphi}^2=2\chi_{\varphi}\wedge \omega \quad \Longleftrightarrow \quad (\chi_{\varphi}-\omega)^2=\omega^2.
\end{align}
Writing $\alpha : = \chi - \omega \ge 0$ on the K\"ahler surface $X$, we can apply Theorem \ref{thm:EGZ1} to see that there exists a  continuous function $\varphi_{\infty}$ solving (\ref{criticaleqn}).  Moreover, $\varphi_{\infty}$ is unique up to the addition of a constant.

Next we use the uniform $C^0$ bound from Theorem \ref{thm:EGZ1} to obtain:

\begin{prop} \label{p1} We assume that $\chi - \omega \ge 0$ as discussed above. Let $\varphi(t)$ be the solution of $J$-flow (\ref{j-flow}) on the compact K\"ahler surface $X$.   Then there exists $C$ depending only on the initial data such that for all $t\geq0$,
\begin{align}
\| \varphi (t) \|_{C^0(X)} \le C.
\end{align}
\end{prop}
\begin{proof} From the introduction, we know
\begin{equation}\label{Cihere}
[\chi-\omega]- \sum_{i=1}^N a_i [C_i] >0,
\end{equation}
 for positive real numbers $a_i$ and irreducible curves $C_i$ of negative self-intersection.  Moreover, since we are assuming $[\chi - \omega] \ge 0$, we may take the constants $a_i$ to be arbitrarily small.   However, we will not need to make use of this last fact.

It follows that there exist Hermitian metrics $h_i$ on the line bundles $[C_i]$ associated to $C_i$ such that
\begin{align} \label{alpham}
\chi-\omega - \sum_{i=1}^N a_i R_{h_i} >0,
\end{align}
where $R_{h_i} = -dd^{c} \log h_i$ is the curvature of $h_i$. Let $s_i$ be a holomorphic section of $[C_i]$ vanishing along $C_i$ to order 1.
Recall that we denote $\chi-\omega$ by $\alpha$.

Next, we apply Theorem \ref{thm:EGZ1} and write $\psi$ for the solution to the degenerate complex Monge-Amp\`ere equation 
\begin{equation} \label{dma2}
(\alpha + dd^c \psi)^2 = \omega^2, \quad \alpha + dd^c \psi \ge 0,
\end{equation}
 subject to the condition $\sup_X \psi =0$.
We have $\| \psi \|_{C^0(X)} \le C$.  

Moreover,  $\psi$ is smooth away from the curves $C_i$.  Indeed this last statement follows from a trick of Tsuji \cite{Ts}, as used in \cite{EGZ1}.  Although the proof is the same, the precise statement we need does not seem to be quite contained in \cite{EGZ1}, and so  we briefly outline the idea here for the convenience of the reader.  For $\delta>0$, let $\psi_{\delta}$ be Yau's solution of the complex Monge-Amp\`ere equation
\begin{equation} \label{psidelta}
(\alpha + \delta \omega + dd^c \psi_{\delta})^2 = c_{\delta} \omega^2, \quad \alpha_{\delta}:= \alpha + \delta \omega+dd^c \psi_{\delta} >0,
\end{equation}
for a constant $c_{\delta}$ chosen so that the integrals of both sides are equal.  From Theorem \ref{thm:EGZ1}, $\psi_{\delta}$ is uniformly bounded in $C^0$.  To obtain a second order estimate for $\psi_{\delta}$, uniform in $\delta$, we consider, for a constant $A>0$,
\begin{equation}
Q_{\delta} =\log \textrm{tr}_{\omega}\alpha_{\delta} - A (\psi_{\delta} - \sum_i a_i \log |s_i|^2_{h_i}),
\end{equation}
which is well-defined on $X \setminus \bigcup C_i$ and tends to $-\infty$ on $\bigcup C_i$.   Compute at a point in $X \setminus \bigcup C_i$,
$$\Delta_{\alpha_{\delta}} Q_{\delta} \ge - C \textrm{tr}_{\alpha_{\delta}} \, \omega - 2A + A \textrm{tr}_{\alpha_{\delta}} (\alpha - \sum_i a_i R_{h_i}).$$
Then using (\ref{alpham}) we may choose a uniform $A$ sufficiently large so that
$$A(\alpha - \sum_i a_i R_{h_i}) \ge (C+1)\omega.$$
The quantity $Q_{\delta}$ achieves a maximum at some point $x \in X \setminus \bigcup C_i$, and at this point we have $\Delta_{\alpha_{\delta}}Q_{\delta} \le 0.$  Hence, at  $x$,
$$0 \ge \textrm{tr}_{\alpha_{\delta}} \, {\omega} -2A,$$
and so  $\textrm{tr}_{\alpha_{\delta}} \, \omega$ is uniformly bounded from above.  But from (\ref{psidelta}) we have at $x$,
$$\textrm{tr}_{\omega} \, {\alpha_{\delta}}  = \left( \frac{\alpha_{\delta}^2}{\omega^2} \right) \textrm{tr}_{\alpha_{\delta}} \, \omega = c_{\delta} \textrm{tr}_{\alpha_{\delta}} \, \omega \le C',$$
for some uniform $C'$.  Since $\psi_{\delta}$ is uniformly bounded in $C^0$ we see that $Q_{\delta}$ is uniformly bounded from above at $x$ and hence everywhere.

This establishes a uniform upper bound for $\textrm{tr}_{\omega} \, \alpha_{\delta}$ (and again by (\ref{psidelta}), also for $\textrm{tr}_{\alpha_{\delta}}\, {\omega}$) on any compact subset of $X \setminus \bigcup C_i$.  It follows that on such fixed compact set, $\omega$ and $\alpha_{\delta}$ are uniformly equivalent.  Hence
 we have 
 estimates, uniform in $\delta$, for $dd^c\psi_{\delta}$ on compact subsets of $X \setminus \bigcup C_i$.  The 
 $C_{\textrm{loc}}^{\infty}(X \setminus \bigcup C_i)$ estimates for $\psi_{\delta}$ then follow from the usual Evans-Krylov \cite{E,Kr} local theory for the complex Monge-Amp\`ere equation.  Taking a limit as $\delta \rightarrow 0$ shows that $\psi$ is smooth away from the $C_i$.

Fix $\ve \in (0,1)$.  We will apply the maximum principle to the quantity 
\begin{align*}
\theta_{\varepsilon} = \varphi - (1+\varepsilon)\psi  + \varepsilon \sum_{i=1}^N a_i \log |s_i|^2_{h_i} - A \varepsilon t,
\end{align*}
where $A$ is a constant to be determined.  Observe that $\theta_{\varepsilon}$ is smooth on $X \setminus \bigcup C_i$ and
tends to negative infinity along $\bigcup C_i$ and hence for each time $t$, $\theta_{\varepsilon}$ achieves a maximum in the interior of $X \setminus \bigcup C_i$.

We rewrite (\ref{j-flow}) as follows
\begin{align} \label{jf2}
\frac{\partial \varphi}{\partial t}&=1-\frac{2\chi_{\varphi}\wedge\omega}{\chi_{\varphi}^2}=\frac{\chi_{\varphi}^2-2\chi_{\varphi}\wedge \omega}{\chi_{\varphi}^2}\\ 
\notag &=\frac{(\chi_{\varphi}-\omega)^2-\omega^2}{\chi_{\varphi}^2}=\frac{\omega^2}{\chi_{\varphi}^2} \left(\frac{(\chi_{\varphi}-\omega)^2}{\omega^2}-1\right)=\frac{\omega^2}{\chi_{\varphi}^2}\left( \frac{\alpha_{\varphi}^2}{\alpha_{\psi}^2}-1\right).
\end{align}
Compute on $X \setminus \bigcup C_i$, using (\ref{jf2}),
\begin{align} \nonumber
\frac{\partial}{\partial t} \theta_{\varepsilon}  & = \frac{\omega^2}{\chi_{\varphi}^2} \left( \frac{(\alpha + dd^{c} \varphi)^2}{(\alpha + dd^{c} \psi)^2} -1 \right) - A \varepsilon \\ \nonumber
& = \frac{\omega^2}{\chi_{\varphi}^2} \left( \frac{ \left( (1+ \varepsilon) \alpha + (1+ \varepsilon)dd^{c} \psi - \varepsilon (\alpha - \sum a_i R_{h_i}) + dd^{c} \theta_{\varepsilon} \right)^2}{(\alpha+dd^{c} \psi)^2} -1 \right) - A\varepsilon.
\end{align}
But $\alpha - \sum a_i R_{h_i} \ge 0$, and at the maximum of $\theta_{\varepsilon}$, we have $dd^{c}\theta_{\varepsilon} \le 0$.  Hence at the maximum of $\theta_{\ve}$,
\begin{align} \label{con}
\frac{\partial}{\partial t} \theta_{\varepsilon}  
& \le  \frac{\omega^2}{\chi_{\varphi}^2} \left((1+\varepsilon)^2 \frac{  \left( \alpha + dd^{c} \psi   \right)^2}{(\alpha+dd^{c} \psi)^2} -1 \right) - A\varepsilon  <0,
\end{align}
if we choose
\[
A = \sup_{X \times [0,\infty)} \frac{3\omega^2}{\chi_{\varphi}^2},
\]
which is a uniform constant since $\chi_{\varphi}$ is always uniformly bounded from below away from zero along the $J$-flow.  Indeed this last fact follows immediately from taking a time derivative of the $J$-flow equation and applying the maximum principle (see Lemma 4.1 in \cite{Ch2}).  Then (\ref{con}) implies that $\theta_{\varepsilon}$ must achieve its maximum at time zero, and hence
$\theta_{\varepsilon}$ is uniformly bounded from above by a constant independent of $\varepsilon$.  Letting $\varepsilon \rightarrow 0$ we obtain the upper bound for $\varphi$.

The lower bound of $\varphi$ is similar:  just replace $\varepsilon$ by $-\varepsilon$ and consider the minimum instead of the maximum.
\end{proof}

We provide a second proof. The proof is based on the equivalence of two notions of weak solution of (\ref{dma}): pluripotential sense and viscosity sense.

\begin{proof} [Second proof of Proposition~\ref{p1}]
As in the first proof, write $\psi$ for the solution to  (\ref{dma2}) with $\sup_X \psi =0$.
The function $\psi$ is continuous on $X$ and is smooth away from the curves $C_i$.
We now apply Theorem 3.6 of \cite{EGZ2}, which states that $\psi$ satisfies the equation  (\ref{dma2}) in the viscosity sense as defined in that paper.

We refer to \cite{EGZ2} for the precise definition of a viscosity solution to (\ref{dma2}).  Instead we state two consequences of this definition which are sufficient for our purposes:\begin{enumerate}
\item[(i)] If $x_0$ is any point on $X$ and $q$ is any smooth function defined in a neighborhood of $x_0$ such that 
$$\psi - q \textrm{ has a local maximum at } x_0$$
then $(\alpha+ dd^c q)^2 \ge \omega^2$ at $x_0$.
\item[(ii)] If $x_0$ is any point on $X$ and $q$ is any smooth function defined in a neighborhood of $x_0$ such that 
$$\psi - q \textrm{ has a local minimum at } x_0$$
then $(\alpha+ dd^c q)^2 \le \omega^2$ at $x_0$.
\end{enumerate}
Indeed (i) follows from the definition of viscosity subsolution and (ii) from the definition of viscosity supersolution (see Section 2 in \cite{EGZ2}).

We first find an upper bound for $\varphi$.  Let $\ve>0$ and define $H_{\ve}=\varphi-\psi - \ve t$.  We wish to show that $H_{\ve}$ attains its maximum value at $t=0$.
 Note that $H_{\ve}$
 satisfies the equation:
\begin{align*}
\frac{\partial H_{\ve}}{\partial t}=1-\frac{2\chi_{\varphi}\wedge\omega}{\chi_{\varphi}^2} - \ve.
\end{align*}
Suppose that $H_{\ve}$ attains a maximum at a point $(x_0, t_0)$ on  $X \times [0,T]$ for some finite $T>0$, and assume for a contradiction that $t_0>0$.  Then $\partial H_{\ve}/{\partial t}\, (x_0,t_0) \ge 0$.  Define a smooth function $q$ on $X$ by $q(x) = \varphi(x,t_0) - H_{\ve}(x_0, t_0) - \ve t_0$.  Note that  the function
$$x\mapsto (\psi -q)(x) = - H_{\ve}(x, t_0) + H_{\ve} (x_0, t_0) $$
achieves its minimum at $x_0$.  Then we can apply (ii) to see that $(\alpha+ dd^c q)^2 \le \omega^2$ at $x_0$, or in other words
\[
(\chi-\omega+ dd^{c} \varphi)^2\leq \omega^2, \quad \text{at} \quad (x_0, t_0),
\]
 which is equivalent to
\[
\chi_{\varphi}^2\leq 2\chi_{\varphi}\wedge \omega \quad \text{at} \quad (x_0, t_0).
\]
It follows that
\[
\frac{\partial H_{\ve}}{\partial t} (x_0, t_0)=1-\frac{2\chi_{\varphi}\wedge \omega}{\chi_{\varphi}^2} - \ve  <0,
\]
contradicting the fact that $\partial H_{\ve}/{\partial t}\, (x_0,t_0) \ge 0$.  Hence $H_{\ve}$ attains its maximum value at $t=0$ and  is uniformly bounded from above independent of $\ve$.  Letting $\ve \rightarrow 0$ gives the desired upper bound for $\varphi$.

Applying a similar argument, using (i) instead of (ii), gives a uniform lower bound for $\varphi$.\end{proof}

We can now apply Theorem 1.3 of \cite{SW} together with the standard local theory for (\ref{j-flow}) to obtain higher order estimates.

\begin{prop} As above, assume that $\chi - \omega \ge 0$ on the compact K\"ahler surface $X$ and let $\varphi(t)$ be the solution of $J$-flow (\ref{j-flow}).  For any compact subset $K\subset X\setminus \bigcup C_i$ and any $k\geq 0$, there exists a constant $C_{k,K}$ such that for all $t$,
\begin{align*}
||\varphi(t)||_{C^k(K)}\leq C_{k,K}.
\end{align*}
Here, the $C_i$ are the irreducible curves of negative self-intersection chosen to satisfy (\ref{Cihere}).
\end{prop}

\section{Proof of the Main Theorem} \label{sect:thm}

Again we assume in this section that $[\chi]$ is scaled so that $c=1$.  Before proving the main theorem we first discuss the $\mathcal{J}$ and $\mathcal{I}$-functionals. Define $\mathcal{J}_{\omega,\chi}$ and $\mathcal{I}_{\omega,\chi}$ by
\begin{align*}
\mathcal{J}_{\omega,\chi}(\varphi):&=\int_{0}^1 \int_X \dot{\varphi_t} (2\chi_{\varphi_t}\wedge \omega-\chi_{\varphi_t}^2) dt, \\
\mathcal{I}_{\omega,\chi}(\varphi):&= \int_{0}^1 \int_X \dot{\varphi_t} \chi_{\varphi_t}^2 dt,
\end{align*}
where $\varphi_t$ is a smooth path in $\mathcal{P}_{\chi}$ connecting $0$ and $\varphi$. For simplicity, we will omit the subscripts.

Note that if $\varphi(t)$ is the solution of the $J$-flow, then
\begin{align}
\frac{d}{dt} \mathcal{J}(\varphi(t))=-\int_X \dot{\varphi}(t)^2 \chi_{\varphi(t)}^2, \quad \frac{d}{dt}\mathcal{I}(\varphi(t))=0.
\end{align}
In particular, the $J$-flow is the gradient flow of $\mathcal{J}$.

One can write explicit formulae for $\mathcal{J}$, $\mathcal{I}$ as follows:
\begin{align} \label{Jform}
\mathcal{J}(\varphi) &= \int_X \varphi(\chi_{\varphi}\wedge \omega +\chi\wedge \omega)-\frac{1}{3}\int_X \varphi(\chi_{\varphi}^2+\chi_{\varphi}\wedge \chi+\chi^2) ,\\ \label{Iform}
\mathcal{I}(\varphi) &=\frac{1}{3} \int_X \varphi(\chi_{\varphi}^2+\chi_{\varphi}\wedge \chi+\chi^2).
\end{align}
Thus an immediate corollary of Proposition~\ref{p1} is:

\begin{prop} There exists a uniform constant $C$ such that, for $\varphi(t)$ the solution of the $J$-flow, we have
\begin{align*}
\mathcal{J}(\varphi(t))\geq -C
\end{align*}
for all $t\geq0$.
\end{prop}

In what follows, we will need to make use of a simple continuity-type result for the $\mathcal{I}$ and $\mathcal{J}$ functionals.

\begin{lem} \label{lem:cty}  Let $\varphi_j \in \mathcal{P}_{\chi}$ and let $\varphi$ be a continuous function on $X$ satisfying $\chi + dd^c \varphi \ge 0$.  Let $Y$ be a proper subvariety of $X$. Suppose that
\begin{enumerate}
\item[(a)] there exists $C$ such that $\| \varphi_j \|_{C^0(X)} \le C$.
\item[(b)] $\varphi_j \rightarrow \varphi$ in $C^{\infty}_{\emph{loc}}(X\setminus Y)$ as $j \rightarrow \infty$.
\end{enumerate}
Then
$$\mathcal{J}(\varphi_j) \rightarrow \mathcal{J}(\varphi) \quad \textrm{and} \quad \mathcal{I}(\varphi_j) \rightarrow \mathcal{I}(\varphi) \quad \textrm{as } j \rightarrow \infty.$$
\end{lem}
\begin{proof}  The proof is a simple exercise in pluripotential theory (we refer the reader to \cite{K2} for an introduction to this theory).  For the convenience of the reader we sketch here the proof. Note first that for $\varphi$ continuous with $\chi+dd^c \varphi \ge 0$, the quantities $\chi_{\varphi}^2$, $\chi \wedge \chi_{\varphi}$ and $\chi_{\varphi} \wedge \omega$ define finite measures on $X$  and hence 
 from (\ref{Jform}) and (\ref{Iform}) the functionals
$\mathcal{I}(\varphi)$ and $\mathcal{J}(\varphi)$ are well-defined.

  We may choose a sequence of open tubular neighborhoods $Y_k$ of $Y$ such that $Y_k \downarrow Y$ as $k \rightarrow \infty$.   Since $Y$ is pluripolar, the capacity 
 $\textrm{Cap}_{\chi}(Y)$ of $Y$ with respect to $\chi$ (in the sense of Kolodziej \cite{K1}) is zero. 
By the properties of this capacity (see \cite{GZ} for example) we have $$\lim_{k \rightarrow \infty} \textrm{Cap}_{\chi}(Y_k) = \textrm{Cap}_{\chi}(Y) =0.$$  
Since the $\varphi_j$ are uniformly bounded, it follows that $\int_{Y_k} \varphi_j \beta \wedge \gamma \rightarrow 0$ as $k \rightarrow \infty$, uniformly in $j$, where $\beta, \gamma$ are each one of $\omega, \chi$ or  $\chi_{\varphi_j}$.  The same holds if we replace  $\varphi_j$ by $\varphi$.  The result then follows from the expressions 
  (\ref{Jform}) and (\ref{Iform}) together with condition (b). 
\end{proof}

Finally we give the proof of the main theorem.

\begin{proof} [Proof of Theorem~\ref{main}]
%First note that the uniqueness statement for $\varphi_{\infty}$ follows from Theorem \ref{thm:EGZ1}, as pointed in Section \ref{sect:main}.

Since $\mathcal{J}$ is decreasing and bounded from below, there exists a constant $C$ such that
\begin{align} \label{intbd}
\int_0^{\infty} \int_X \dot{\varphi}(t)^2 \chi_{\varphi(t)}^2 dt < C.
\end{align}
We claim that for each fixed point $p \in X \setminus \bigcup C_i$,
 $\dot{\varphi}(p, t) \rightarrow 0$ as $t \rightarrow \infty$.  Indeed, suppose not.  Then there exists $\varepsilon >0$ and a sequence of times $t_i \rightarrow \infty$ such that
 $|\dot{\varphi} (t_i)| >\varepsilon $ for all $i$.  But since we have bounds for $\dot{\varphi}$ and all its time and space derivatives in a fixed neighborhood $U$, say, of $p$ with $U \subset X \setminus \bigcup C_i$, it follows that $| \dot{\varphi} (t)| > \varepsilon /2$ for $t \in [t_i, t_i+\delta]$ for a uniform $\delta>0$.  This contradicts (\ref{intbd}) and establishes the claim.

Since we have $C^{\infty}_{\textrm{loc}} (X \setminus \bigcup C_i)$ bounds for $\dot{\varphi}$,  the uniqueness of limits implies that $\dot{\varphi}$ converges to zero in  $C^{\infty}_{\textrm{loc}} (X \setminus \bigcup C_i)$.

We have uniform $C^{\infty}$ bounds for $\varphi(t)$ on compact subsets of $X \setminus \bigcup C_i$, and hence we can apply the Arzela-Ascoli theorem to see that for a sequence of times $t_i \rightarrow \infty$ we have $\varphi(t_i) \rightarrow \varphi_{\infty}$ for a smooth (bounded) function $\varphi_{\infty}$ on $X \setminus \bigcup C_i$.  Moreover, since $\dot{\varphi} \rightarrow 0$, $\varphi_{\infty}$ satisfies the equation $\chi_{\varphi_{\infty}}^2=2 \chi_{\varphi_{\infty}} \wedge \omega  $ as in the statement of the theorem.

We also have $\mathcal{I}(\varphi_{\infty})=\lim_{t \rightarrow \infty} \mathcal{I}( \varphi(t)) = \mathcal{I}(\varphi_0)$, using Lemma \ref{lem:cty} and the fact that $\mathcal{I}$ is constant along the flow.  Applying Theorem \ref{thm:EGZ1}, we know that (\ref{criticaleqn}) has a unique solution up to the addition of a constant.  Thus $\varphi_{\infty}$ is the unique solution of (\ref{criticaleqn}) subject to the condition $\mathcal{I}(\varphi_{\infty}) =\mathcal{I}(\varphi_0)$.

Finally we claim that $\varphi(t)$ converges in $C^{\infty}_{\textrm{loc}} (X \setminus \bigcup C_i)$ to $\varphi_{\infty}$.  Suppose not.  Then there exists $\varepsilon >0$ and a sequence of times $t_i \rightarrow \infty$ such that
$\| \varphi(t_i) - \varphi_{\infty} \|_{C^k(K)} > \varepsilon $ for all $i$, for some integer $k$ and compact  $K \subset X \setminus \bigcup C_i$.  Since we have uniform $C^{\infty}$ bounds for $\varphi(t)$ on $K$ we can pass to a subsequence and assume that $\varphi(t_i)$ converges to a function $\varphi'_{\infty} \neq \varphi_{\infty}$.  But  $\varphi'_{\infty}$ will also satisfy the equation
$ {\chi}_{\varphi'_{\infty}}^2=2 \chi_{\varphi'_{\infty}} \wedge \omega $ and $\mathcal{I}(\varphi'_{\infty}) =\mathcal{I}(\varphi_0)$ contradicting the uniqueness.
\end{proof}

As a consequence:

\begin{cor} \label{c1}
The $\mathcal{J}$-functional is bounded from below on $\mathcal{P}_{\chi}$.
\end{cor}
\begin{proof}
Take any $\varphi_{_0} \in \mathcal{P}_{\chi}$.   Then running the $J$-flow from $\varphi_{_0}$, which by Theorem \ref{main} converges to $\varphi_{\infty}$, we obtain (applying Lemma \ref{lem:cty}),
\begin{align*}
\mathcal{J}(\varphi_{_0}) \ge \lim_{t \rightarrow \infty} \mathcal{J}(\varphi(t)) = \mathcal{J} (\varphi_{\infty}),
\end{align*}
since $\mathcal{J}$ is decreasing along the flow.
\end{proof}

Finally:

\begin{proof}[Proof of Corollary \ref{t2}]
Combine Corollary ~\ref{c1} and Lemma 4.1 of ~\cite{SW}.
\end{proof}

\section{Further questions}

Finally, we propose some further questions.

\begin{question} \label{q1}
In general, it does not appear to be known whether a nef and big class on a K\"ahler surface can always be represented by a smooth nonnegative $(1,1)$ form (for a counterexample in higher dimensions see Example 5.4 in \cite{BEGZ}).  However, an example of Zariski shows that a nef and big class is not necessarily semiample (see Section 2.3A of \cite{Laz}).   Also,  the nef condition alone is not sufficient for the existence of a nonnegative representative  (see Example 1.7 of ~\cite{DPS}).
What can be proved if we assume only that $[\chi-\omega]$ is nef and big? In this case, by~\cite{BEGZ}, we know that we can produce a solution  $\psi$ of (\ref{dma}) with very mild singularities along $C_i$ (less than any log pole). Can it be translated into an estimate for the solution $\varphi(t)$ of the $J$-flow? Does it imply that the $J$-functional is bounded from below?
\end{question}

\begin{question} The results of ~\cite{FL2} indicate a possible picture when $[\chi]$ is outside of $\mathcal{C}_{\omega}$. But they assume both $\omega$ and $\chi$ are of Calabi ansatz. Can one prove a general result on K\"{a}hler surfaces? In this case, presumably the $\mathcal{J}$-functional is not bounded from below.
\end{question}

\begin{question}
For general $n$, it would be interesting to investigate the weak solution of the critical equation (\ref{critical}) when $[\chi]$ does not lie in  $\mathcal{C}_{\omega}$.
\end{question}

\end{document}